\documentclass[11pt]{amsart}
\usepackage{amssymb}
\usepackage{amsmath}
\usepackage{amsfonts}
\usepackage[active]{srcltx}

\newtheorem{theorem}{Theorem}

\newtheorem{definition}{Definition}
\newtheorem{corollary}{Corollary}
\newtheorem*{Dr}{Theorem D}
\newtheorem*{GS}{Theorem GS1}
\newtheorem*{GS2}{Theorem GS2}

\begin{document}
\author{Ushangi Goginava }
\address{U. Goginava, Department of Mathematics, Faculty of Exact and
Natural Sciences, Tbilisi State University, Chavcha\-vadze str. 1, Tbilisi
0128, Georgia}
\email{zazagoginava@gmail.com}
\title[Convergence of Fourier-Legendre series]{Uniform Convergence of Double
Fourier-Legendre series of Functions of Bounded Generalized Variation }
\date{}
\maketitle

\begin{abstract}
The Uniform convergence of double Fourier-Legendre series of function of
bounded Harmonic variation and bounded partial $\Lambda $-variation are
investigated.
\end{abstract}

\footnotetext{%
2010 Mathematics Subject Classification 42C10 .
}

\section{Classes of Functions of Bounded Generalized Variation}

In 1881 Jordan \cite{Jo} introduced a class of functions of bounded
variation and applied it to the theory of Fourier series. Hereinafter this
notion was generalized by many authors (quadratic variation, $\Phi $%
-variation, $\Lambda $-variation ets., see \cite{Ch, M, Wi,W}). In two
dimensional case the class BV of functions of bounded variation was
introduced by Hardy \cite{Ha}.

Let $f$ be a real function of two variable. Given intervals $\Delta =(a,b)$,
$J=(c,d)$ and points $x,y$ from $I:=[-1,1]$ we denote
\begin{equation*}
f(\Delta ,y):=f(b,y)-f(a,y),\qquad f(x,J)=f(x,d)-f(x,c)
\end{equation*}%
and
\begin{equation*}
f(\Delta ,J):=f(a,c)-f(a,d)-f(b,c)+f(b,d).
\end{equation*}%
Let $E=\{\Delta _{i}\}$ be a collection of nonoverlapping intervals from $I$
ordered in arbitrary way and let $\Omega $ be the set of all such
collections $E$. Denote by $\Omega _{n}$ set of all collections of $n$
nonoverlapping intervals $\Delta _{k}\subset I.$

For the sequence of positive numbers $\Lambda =\{\lambda
_{n}\}_{n=1}^{\infty }$ and $I^{2}:=[-1,1]^{2}$ we denote
\begin{equation*}
\Lambda V_{1}(f;I^{2})=\sup_{y}\sup_{E\in \Omega }\sum_{i}\frac{|f(\Delta
_{i},y)|}{\lambda _{i}}\,\,\,\,\,\,\left( E=\{\Delta _{i}\}\right) ,
\end{equation*}%
\begin{equation*}
\Lambda V_{2}(f;I^{2})=\sup_{x}\sup_{F\in \Omega }\sum_{j}\frac{|f(x,J_{j})|%
}{\lambda _{j}}\qquad (F=\{J_{j}\}),
\end{equation*}%
\begin{equation*}
\Lambda V_{1,2}(f;I^{2})=\sup_{F,\,E\in \Omega }\sum_{i}\sum_{j}\frac{%
|f(\Delta _{i},J_{j})|}{\lambda _{i}\lambda _{j}}.
\end{equation*}

\begin{definition}
We say that the function $f$ has Bounded $\Lambda $-variation on $I^{2}$ and
write $f\in \Lambda BV$, if
\begin{equation*}
\Lambda V(f;I^{2}):=\Lambda V_{1}(f;I^{2})+\Lambda V_{2}(f;I^{2})+\Lambda
V_{1,2}(f;I^{2})<\infty .
\end{equation*}%
We say that the function $f$ has Bounded Partial $\Lambda $-variation and
write $f\in P\Lambda BV$ if
\begin{equation*}
P\Lambda V(f;I^{2}):=\Lambda V_{1}(f;I^{2})+\Lambda V_{2}(f;I^{2})<\infty .
\end{equation*}
\end{definition}

\begin{definition}
We say that the function $f\,$\ is continuous in $\left( \Lambda
^{1},\Lambda ^{2}\right) $-variation on $I^{2}$ and write $f\in C\left(
\Lambda ^{1},\Lambda ^{2}\right) V\left( I^{2}\right) $, if%
\begin{equation*}
\lim\limits_{n\rightarrow \infty }\Lambda _{n}^{1}V_{1}\left( f;I^{2}\right)
=\lim\limits_{n\rightarrow \infty }\Lambda _{n}^{2}V_{2}\left(
f;I^{2}\right) =0
\end{equation*}%
and%
\begin{equation*}
\lim\limits_{n\rightarrow \infty }\left( \Lambda _{n}^{1},\Lambda
^{2}\right) V_{1,2}\left( f;I^{2}\right) =\lim\limits_{n\rightarrow \infty
}\left( \Lambda ^{1},\Lambda _{n}^{2}\right) V_{1,2}\left( f;I^{2}\right) =0,
\end{equation*}%
where $\Lambda _{n}^{i}:=\left\{ \lambda _{k}^{i}\right\} _{k=n}^{\infty
}=\left\{ \lambda _{k+n}^{i}\right\} _{k=0}^{\infty },i=1,2.$
\end{definition}

If $\lambda _{n}\equiv 1$ (or if $0<c<\lambda _{n}<C<\infty ,\ n=1,2,\ldots $%
) the classes $\Lambda BV$ and $P\Lambda BV$ coincide with the Hardy class $%
BV$ and PBV respectively. Hence it is reasonable to assume that $\lambda
_{n}\rightarrow \infty $ and since the intervals in $E=\{\Delta _{i}\}$ are
ordered arbitrarily, we will suppose, without loss of generality, that the
sequence $\{\lambda _{n}\}$ is increasing. Thus,
\begin{equation}
1<\lambda _{1}\leq \lambda _{2}\leq \ldots ,\qquad \lim_{n\rightarrow \infty
}\lambda _{n}=\infty .  \label{Lambda}
\end{equation}

In the case when $\lambda _{n}=n,\ n=1,2\ldots $ we say \textit{Harmonic
Variation} instead of $\Lambda $-variation and write $H$ instead of $\Lambda$
($HBV$, $PHBV$, $HV(f)$, ets).

The notion of $\Lambda $-variation was introduced by Waterman \cite{W} in
one dimensional case, by Sahakian \cite{Saha} in two dimensional case. The
notion of bounded partial variation ($PBV$) was introduced by Goginava \cite%
{GoEJA} and the notion of bounded partial $\Lambda $-variation, by Goginava
and Sahakian \cite{GogSaha}.

The statements of the following theorem is known.

\begin{Dr}[Dragoshanski \protect\cite{Drag}]
$HBV=CHV$.
\end{Dr}

\begin{definition}
Let $\Phi $-be a strictly increasing continuous function on $[0,+\infty )$
with $\Phi \left( 0\right) =0$. We say that the function $f$ has bounded
partial $\Phi $-variation on $I^{2}$ and write $f\in PBV_{\Phi }$, if
\begin{equation*}
V_{\Phi }^{\left( 1\right) }\left( f\right)
:=\sup\limits_{y}\sup\limits_{\{I_{i}\}\in \Omega
_{n}}\sum\limits_{i=1}^{n}\Phi \left( |f\left( I_{i},y\right) |\right)
<\infty ,\quad n=1,2,...,
\end{equation*}%
\begin{equation*}
V_{\Phi }^{\left( 2\right) }\left( f\right)
:=\sup\limits_{x}\sup\limits_{\{J_{j}\}\in \Omega
_{m}}\sum\limits_{j=1}^{m}\Phi \left( |f\left( x,J_{j}\right) |\right)
<\infty ,\quad m=1,2,....
\end{equation*}
\end{definition}

In the case when $\Phi \left( u\right) =u^{p},\,p\geq 1$, the notion of
bounded partial $p$-variation (class $PBV_{p}$) was introduced in \cite%
{GoJAT}.

In \cite{GogSaha} it is proved that the following theorem is true.

\begin{GS}
\label{GS1} Let $\Lambda =\{\lambda _{n}\}$ and $\lambda _{n}/n\geq \lambda
_{n+1}/\left( n+1\right) >0,\ n=1,2,....\,\,\,$. \newline
1) If
\begin{equation*}
\sum_{n=1}^{\infty }\frac{\lambda _{n}}{n^{2}}<\infty ,
\end{equation*}%
then $P\Lambda BV\subset HBV$.\newline
2) If, in addition, for some $\delta >0$
\begin{equation*}
\frac{\lambda _{n}}{n}=O(\frac{\lambda _{n^{[1+\delta ]}}}{n^{[1+\delta ]}}%
)\quad \text{as}\quad n\rightarrow \infty
\end{equation*}%
and
\begin{equation*}
\sum_{n=1}^{\infty }\frac{\lambda _{n}}{n^{2}}=\infty ,
\end{equation*}%
then $P\Lambda BV\not\subset HBV$.
\end{GS}

\begin{corollary}
$PBV\subset HBV$ and $PHBV\not\subset HBV$.
\end{corollary}

\begin{definition}[see \protect\cite{GoEJA}]
The partial modulus of variation of a function $f$ are the functions $%
v_{1}\left( n,f\right)$ and $v_{2}\left( m,f\right) $ defined by
\begin{equation*}
v_{1}\left( n,f\right) :=\sup\limits_{y}\sup\limits_{\{I_{i}\}\in \Omega
_{n}}\sum\limits_{i=1}^{n}\left| f\left( I_{i},y\right) \right| ,\quad
n=1,2,\ldots,
\end{equation*}
\begin{equation*}
v_{2}\left( m,f\right) :=\sup\limits_{x}\sup\limits_{\{J_{k}\}\in \Omega
_{m}}\sum\limits_{i=1}^{m}\left| f\left( x,J_{k}\right) \right| ,\quad
m=1,2,\ldots.
\end{equation*}
\end{definition}

For functions of one variable the concept of the modulus variation was
introduced by Chanturia \cite{Ch}.

The following theorem is proved by Goginava and Sahakian \cite{GogSaha}.

\begin{GS2}
\label{GS2}If $f\in B$ is bounded on $I^{2}$ and
\begin{equation*}
\sum\limits_{n=1}^{\infty }\frac{\sqrt{v_{j}\left( n,f\right) }}{n^{3/2}}%
<\infty ,\quad j=1,2,
\end{equation*}%
then $f\in HBV.$
\end{GS2}

$C\left( I^{2}\right) $ is the space of continuous functions on $I^{2}$ with
uniform norm%
\begin{equation*}
\left\Vert f\right\Vert _{C}:=\max\limits_{\left( x,y\right) \in
I^{2}}\left\vert f\left( x,y\right) \right\vert .
\end{equation*}

The partial moduli of continuity of a function $f\in C\left( I^{2}\right) $
in the $C$-norm are defined by%
\begin{equation*}
\omega _{1}\left( f;\delta \right) :=\max \left\{ \left\vert f\left(
x,y\right) -f\left( s,y\right) \right\vert ,x,y,s\in I,\left\vert
x-s\right\vert \leq \delta \right\} ,
\end{equation*}%
\begin{equation*}
\omega _{2}\left( f;\delta \right) :=\max \left\{ \left\vert f\left(
x,y\right) -f\left( x,t\right) \right\vert ,x,y,t\in I,\left\vert
y-t\right\vert \leq \delta \right\} ,
\end{equation*}%
while the mixed modulus of continuity is defined as follows:%
\begin{eqnarray*}
\omega _{1,2}\left( f;\delta _{1},\delta _{2}\right) &:&=\max \left\{
\left\vert f\left( x,y\right) -f\left( s,y\right) -f\left( x,t\right)
+f\left( s,t\right) \right\vert ,\right. \\
&&\left. x,y,s,t\in I,\left\vert x-s\right\vert \leq \delta _{1},\left\vert
y-t\right\vert \leq \delta _{2}\right\} .
\end{eqnarray*}

\section{Fourier-Legendre Series}

Let $p_{n}\left( x\right) $ be the Legendre orthonormal polynomial of degree
$n$. If $f$ is an integrable function on $I^{2}$, then Fourier-Legendre
series of $f$ is the series%
\begin{equation*}
\sum\limits_{n=0}^{\infty }\sum\limits_{m=0}^{\infty }\widehat{f}\left(
n,m\right) p_{n}\left( x\right) p_{n}\left( y\right) ,
\end{equation*}%
where%
\begin{equation*}
\widehat{f}\left( n,m\right)
:=\int\limits_{-1}^{1}\int\limits_{-1}^{1}f\left( s,t\right) p_{n}\left(
s\right) p_{m}\left( t\right) dsdt
\end{equation*}%
is $\left( n,m\right) $th Fourier coefficient of the function $f$.

The rectangular partial sums of double Fourier-Legendre series are defined by%
\begin{equation*}
S_{MN}f\left( x,y\right) :=\sum\limits_{n=0}^{N-1}\sum\limits_{m=0}^{M-1}%
\widehat{f}\left( n,m\right) p_{n}\left( x\right) p_{m}\left( y\right) .
\end{equation*}

It is easy to show that%
\begin{equation}
S_{MN}f\left( x,y\right) =\int\limits_{-1}^{1}\int\limits_{-1}^{1}f\left(
s,t\right) K_{n}\left( x,s\right) K_{m}\left( y,t\right) dsdt,  \label{PS}
\end{equation}%
where%
\begin{equation*}
K_{n}\left( x,s\right) :=\sum\limits_{k=0}^{n-1}p_{k}\left( s\right)
p_{k}\left( x\right) .
\end{equation*}

It is well-know the Chrestoffel-Darboux formula (see (\cite{Su}))%
\begin{equation}
K_{n}\left( x,t\right) =\frac{\gamma _{n-1}}{\gamma _{n}}\frac{p_{n-1}\left(
t\right) p_{n}\left( x\right) -p_{n-1}\left( x\right) p_{n}\left( t\right) }{%
x-t}.  \label{CD}
\end{equation}

Since%
\begin{equation*}
\frac{\gamma _{n-1}}{\gamma _{n}}\leq 1
\end{equation*}%
and%
\begin{equation}
\left\vert p_{n}\left( x\right) \right\vert \leq \frac{c}{\left(
1-x^{2}\right) ^{1/4}},x\in \left( -1,1\right)  \label{p1}
\end{equation}%
from (\ref{CD}) we have%
\begin{equation}
\left\vert K_{n}\left( x,t\right) \right\vert \leq \frac{c}{\left\vert
x-t\right\vert \left( 1-x^{2}\right) ^{1/4}\left( 1-t^{2}\right) ^{1/4}}.
\label{Kn}
\end{equation}

In \cite{Boj, Pow} it is proved that the following estimations holds%
\begin{equation}
\left\vert \int\limits_{-1}^{s}K_{n}\left( x,t\right) dt\right\vert \leq
\frac{c}{n\left( x-s\right) \left( 1-x^{2}\right) ^{1/4}}~\ \ \left( -1\leq
s<x<1\right) ~,  \label{Kn_}
\end{equation}%
\begin{equation}
\left\vert \int\limits_{s}^{1}K_{n}\left( x,t\right) dt\right\vert \leq
\frac{c}{n\left( s-x\right) \left( 1-x^{2}\right) ^{1/4}}~\ \ \left(
-1<x<s\leq 1\right) ,  \label{Kn^}
\end{equation}%
\begin{equation}
\int\limits_{x-\frac{1+x}{n}}^{x}\left\vert K_{n}\left( x,t\right)
\right\vert dt\leq \frac{c\left( 1+x\right) }{\left( 1-x^{2}\right) ^{1/2}}%
~\ \ \ \left( -1<x<1\right) ,  \label{Kn^x}
\end{equation}%
\begin{equation}
\int\limits_{x}^{x+\frac{1-x}{n}}\left\vert K_{n}\left( x,t\right)
\right\vert dt\leq \frac{c\left( 1-x\right) }{\left( 1-x^{2}\right) ^{1/2}}%
~\ \ \ \left( -1<x<1\right) .  \label{Kn_x}
\end{equation}

\section{Convergence of double Fourier-Legendre series}

The well known Dirichlet-Jordan theorem (see \cite{Zy}) states that the
trigonometric Fourier series of a function $f(x),\ x\in \lbrack 0,2\pi )$ of
bounded variation converges at every point $x$ to the value $\left[ f\left(
x+0\right) +f\left( x-0\right) \right] /2$.

Hardy \cite{Ha} generalized the Dirichlet-Jordan theorem to the double
trigonometric Fourier series. He proved that if function $f(x,y)$ has
bounded variation in the sense of Hardy ($f\in BV$), then double
trigonometric Fourier series of the continuous function $f$ converges
uniformly on $\left[ 0,2\pi \right] ^{2}$. The author \cite{GoEJA} has
proved that in Hardy's theorem there is no need to require the boundedness
of $V_{1,2}(f)$; moreover, it is proved that if $f$ is continuous function
and has bounded partial variation $\left( f\in PBV\right) $ then its double
trigonometric Fourier series converges uniformly on $[0,2\pi ]^{2}$.

Convergence of rectangular and spherical partial sums of d-dimensional
trigonometric Fourier series of functions of bounded $\Lambda $-variation
was investigated in details by Sahakian \cite{Saha}, Dyachenko \cite{D1, D2,
DW}, Bakhvalov \cite{Bakh}, Sablin \cite{Sab}.

For the one-dimensional Fourier-Legendre series the convergence of partial
sums of functions Harmonic bounded variation and other bounded generalized
variation were studied by Agakhanov, Natanson \cite{AN}, Bojanic \cite{Boj},
Belenki\u{\i} \cite{Bel}, Kvernadze \cite{Kv1, Kv2, Kv3}, Powierska \cite%
{Pow}.

In this paper we prove that the following are true.

\begin{theorem}
\label{CBV}Let $\varepsilon >0$ and $f$ be a function of $CHV\left(
I^{2}\right) \bigcap C\left( I^{2}\right) $. Then double Fourier-Legendre
series of the function $f$ uniformly converges to $f$ on $\left[
-1+\varepsilon ,1-\varepsilon \right] ^{2}$.
\end{theorem}

Theorem D and Theorem \ref{CBV} imply

\begin{theorem}
\label{HBV}Let $\varepsilon >0$ and $f$ be a function of $HBV\left(
I^{2}\right) \bigcap C\left( I^{2}\right) $. Then double Fourier-Legendre
series of the function $f$ uniformly converges to $f$ on $\left[
-1+\varepsilon ,1-\varepsilon \right] ^{2}$.
\end{theorem}

Theorem GS1 and Theorem \ref{HBV} imply

\begin{theorem}
\label{PBV}Let $f\in P\Lambda BV(I^{2})\bigcap C\left( I^{2}\right) $ with
\begin{equation*}
\sum\limits_{j=1}^{\infty }\frac{\lambda _{j}}{j^{2}}<\infty ,\quad \frac{%
\lambda _{j}}{j}\downarrow 0.
\end{equation*}%
Then double Fourier-Legendre series of the function $f$ uniformly converges
to $f$ on $\left[ -1+\varepsilon ,1-\varepsilon \right] ^{2}$,$\varepsilon
>0.$
\end{theorem}

\begin{corollary}
\label{CPBV}If $f\in P\left\{ \frac{n}{\log ^{1+\delta }\left( n+1\right) }%
\right\} BV\left( I^{2}\right) \bigcap C\left( I^{2}\right) $ for some $%
\delta >0$. Then double Fourier-Legendre series of the function $f$
uniformly converges to $f$ on $\left[ -1+\varepsilon ,1-\varepsilon \right]
^{2}$.
\end{corollary}

Theorem GS2 and Theorem \ref{HBV} imply

\begin{theorem}
\label{PCV}Let $f\in C\left( I^{2}\right) $ and%
\begin{equation*}
\sum\limits_{n=1}^{\infty }\frac{\sqrt{v_{j}\left( n,f\right) }}{n^{3/2}}%
<\infty ~\ \ \ \ j=1,2.
\end{equation*}%
Then double Fourier-Legendre series of the function $f$ uniformly converges
to $f$ on $\left[ -1+\varepsilon ,1-\varepsilon \right] ^{2}$,$\varepsilon
>0.$
\end{theorem}

\begin{corollary}
\label{CPCV1}Let $f\in f\in C\left( I^{2}\right) $ and $v_{1}\left(
k,f\right) =O\left( k^{\alpha }\right) ,v_{2}\left( k,f\right) =O\left(
k^{\beta }\right) ,0<\alpha ,\beta <1$. Then double Fourier-Legendre series
of the function $f$ uniformly converges to $f$ on $\left[ -1+\varepsilon
,1-\varepsilon \right] ^{2}$,$\varepsilon >0.$
\end{corollary}

\begin{corollary}
\label{CPCV2}Let $f\in PBV_{p}\bigcap C\left( I^{2}\right) ,p\geq 1$. Then
double Fourier-Legendre series of the function $f$ uniformly converges to $f$
on $\left[ -1+\varepsilon ,1-\varepsilon \right] ^{2}$,$\varepsilon >0.$
\end{corollary}

\section{Proofs of Main Results}

\begin{proof}[Proof of Theorem \protect\ref{CBV}]
Denote%
\begin{equation}
s_{j}:=x+\frac{j\left( 1-x\right) }{n},j=1,2,...,n,x\in \left( -1,1\right) ,
\label{sj}
\end{equation}%
\begin{equation}
t_{i}:=y-\frac{i\left( 1+y\right) }{m},i=1,2,...,m,y\in \left( -1,1\right)
\label{ti}
\end{equation}%
\begin{equation}
g\left( s,t\right) :=f\left( s,t\right) -f\left( x,y\right) .  \label{g}
\end{equation}

Then from (\ref{PS}) we can write%
\begin{eqnarray}
&&S_{mn}f\left( x,y\right) -f\left( x,y\right)  \label{I1-I4} \\
&=&\int\limits_{-1}^{1}\int\limits_{-1}^{1}g\left( s,t\right) K_{n}\left(
x,s\right) K_{m}\left( y,t\right) dsdt  \notag \\
&=&\left(
\int\limits_{x}^{1}\int\limits_{-1}^{y}+\int\limits_{-1}^{x}\int%
\limits_{-1}^{y}+\int\limits_{-1}^{x}\int\limits_{y}^{1}+\int\limits_{x}^{1}%
\int\limits_{y}^{1}\right) \left( g\left( s,t\right) K_{n}\left( x,s\right)
K_{m}\left( y,t\right) dsdt\right)  \notag \\
&=&I_{1}+I_{2}+I_{3}+I_{4},  \notag
\end{eqnarray}%
\begin{eqnarray}
&&I_{1}  \label{II1-II4} \\
&=&\left(
\int\limits_{x}^{s_{1}}\int\limits_{-1}^{t_{1}}+\int\limits_{x}^{s_{1}}\int%
\limits_{t_{1}}^{y}+\int\limits_{s_{1}}^{1}\int\limits_{-1}^{t_{1}}+\int%
\limits_{s_{1}}^{1}\int\limits_{t_{1}}^{y}\right) \left( g\left( s,t\right)
K_{n}\left( x,s\right) K_{m}\left( y,t\right) dsdt\right)  \notag \\
&=&II_{1}+II_{2}+II_{3}+II_{4}.  \notag
\end{eqnarray}

For $II_{4}$ we have%
\begin{eqnarray}
II_{4} &=&\int\limits_{t_{1}}^{y}\left(
\sum\limits_{j=1}^{n-2}\int\limits_{s_{j}}^{s_{j+1}}\left( g\left(
s,t\right) -g\left( s_{j},t\right) \right) K_{n}\left( x,s\right) ds\right)
K_{m}\left( y,t\right) dt  \label{II41-II43} \\
&&+\int\limits_{t_{1}}^{y}\left( \int\limits_{s_{n-1}}^{1}\left( g\left(
s,t\right) -g\left( s_{n-1},t\right) \right) K_{n}\left( x,s\right)
ds\right) K_{m}\left( y,t\right) dt  \notag \\
&&+\int\limits_{t_{1}}^{y}\left( \sum\limits_{j=1}^{n-1}g\left(
s_{j},t\right) \int\limits_{s_{j}}^{s_{j+1}}K_{n}\left( x,s\right) ds\right)
K_{m}\left( y,t\right) dt  \notag \\
&=&II_{41}+II_{42}+II_{43}.  \notag
\end{eqnarray}

From (\ref{p1}) and (\ref{Kn}) we have%
\begin{eqnarray}
\left\vert II_{42}\right\vert  &\leq &2\left\Vert f\right\Vert
_{C}\int\limits_{t_{1}}^{y}\sum\limits_{j=0}^{m-1}\left\vert p_{j}\left(
t\right) p_{j}\left( y\right) \right\vert
dt\int\limits_{s_{n-1}}^{1}\left\vert K_{n}\left( x,s\right) \right\vert ds
\label{II42} \\
&\leq &c\left\Vert f\right\Vert _{C}~\int\limits_{y-\frac{1+y}{m}}^{y}\frac{%
mdt}{\left( 1-t^{2}\right) ^{1/4}\left( 1-y^{2}\right) ^{1/4}}  \notag \\
&&\times \int\limits_{s_{n-1}}^{1}\frac{ds}{\left( s-x\right) \left(
1-x^{2}\right) ^{1/4}\left( 1-s^{2}\right) ^{1/4}}  \notag \\
&\leq &c\left( \varepsilon \right) \left\Vert f\right\Vert
_{C}\int\limits_{x+\frac{n-1}{n}\left( 1-x\right) }^{1}\frac{ds}{\left(
1-s\right) ^{1/4}}  \notag \\
&\leq &\frac{c\left( \varepsilon \right) \left\Vert f\right\Vert _{C}}{%
n^{3/4}}=o\left( 1\right) \text{ \ \ as \ }n,m\rightarrow \infty   \notag
\end{eqnarray}%
uniformly with respect to $\left( x,y\right) \in \left[ -1+\varepsilon
,1-\varepsilon \right] ^{2}.$

From (\ref{Kn}), (\ref{sj}) and (\ref{ti}) we obtain%
\begin{eqnarray}
&&\left\vert II_{41}\right\vert \label{II41a} \\
&\leq &\int\limits_{t_{1}}^{y}\left(
\sum\limits_{j=1}^{n-2}\int\limits_{s_{j}}^{s_{j+1}}\frac{\left\vert g\left(
s,t\right) -g\left( s_{j},t\right) \right\vert }{\left( s-x\right) \left(
1-x^{2}\right) ^{1/4}\left( 1-s^{2}\right) ^{1/4}}ds\right) \left\vert
K_{m}\left( y,t\right) \right\vert dt  \notag   \\
&\leq &c\left( \varepsilon \right) m\int\limits_{t_{1}}^{y}\left(
\sum\limits_{j=1}^{n-2}\int\limits_{s_{j}}^{s_{j+1}}\frac{\left\vert g\left(
s,t\right) -g\left( s_{j},t\right) \right\vert }{\left( s_{j}-x\right)
\left( 1-s_{j+1}\right) ^{1/4}}ds\right) dt  \notag \\
&\leq &c\left( \varepsilon \right) n^{5/4}m\int\limits_{t_{1}}^{y}\left(
\sum\limits_{j=1}^{n-2}\frac{1}{j\left( n-j\right) ^{1/4}}%
\int\limits_{s_{j}}^{s_{j+1}}\left\vert g\left( s,t\right) -g\left(
s_{j},t\right) \right\vert ds\right) dt  \notag \\
&=&c\left( \varepsilon \right) n^{5/4}m\int\limits_{t_{1}}^{y}\left(
\int\limits_{0}^{\frac{1-x}{n}}\sum\limits_{j=1}^{n-2}\frac{\left\vert
g\left( s+s_{j},t\right) -g\left( s_{j},t\right) \right\vert }{j\left(
n-j\right) ^{1/4}}ds\right) dt  \notag \\
&=&c\left( \varepsilon \right) n^{5/4}m\int\limits_{t_{1}}^{y}\left(
\int\limits_{0}^{\frac{1-x}{n}}\sum\limits_{1\leq j<n/2}\frac{\left\vert
g\left( s+s_{j},t\right) -g\left( s_{j},t\right) \right\vert }{j\left(
n-j\right) ^{1/4}}ds\right) dt  \notag \\
&&+c\left( \varepsilon \right) n^{5/4}m\int\limits_{t_{1}}^{y}\left(
\int\limits_{0}^{\frac{1-x}{n}}\sum\limits_{n/2\leq j<n-1}\frac{\left\vert
g\left( s+s_{j},t\right) -g\left( s_{j},t\right) \right\vert }{j\left(
n-j\right) ^{1/4}}ds\right) dt  \notag \\
&\leq &c\left( \varepsilon \right) nm\int\limits_{t_{1}}^{y}\left(
\int\limits_{0}^{\frac{1-x}{n}}\sum\limits_{1\leq j<n/2}\frac{\left\vert
g\left( s+s_{j},t\right) -g\left( s_{j},t\right) \right\vert }{j}ds\right) dt
\notag \\
&&+c\left( \varepsilon \right) n^{1/4}m\int\limits_{t_{1}}^{y}\left(
\int\limits_{0}^{\frac{1-x}{n}}\sum\limits_{n/2\leq j<n-1}\frac{\left\vert
g\left( s+s_{j},t\right) -g\left( s_{j},t\right) \right\vert }{\left(
n-j\right) ^{1/4}}ds\right) dt.  \notag
\end{eqnarray}

It is easy to show that%
\begin{eqnarray}
&&\sum\limits_{1\leq j<n/2}\frac{\left\vert g\left( s+s_{j},t\right)
-g\left( s_{j},t\right) \right\vert }{j}  \label{1} \\
&\leq &\min\limits_{1\leq k<n}\left\{ \sum\limits_{1\leq j<k}\frac{%
\left\vert g\left( s+s_{j},t\right) -g\left( s_{j},t\right) \right\vert }{j}%
+\sum\limits_{k\leq j<n}\frac{\left\vert g\left( s+s_{j},t\right) -g\left(
s_{j},t\right) \right\vert }{j}\right\}  \notag \\
&\leq &c\left( \varepsilon \right) \min\limits_{1\leq k<n}\left\{ \omega
_{1}\left( f;\frac{1}{n}\right) \log \left( k+1\right) +\left\{ j+k\right\}
V_{1}\left( f;I^{2}\right) \right\}  \notag \\
&=&o\left( 1\right) \text{ \ \ as \ \ }n\rightarrow \infty  \notag
\end{eqnarray}%
uniformly with respect to $\left( x,y\right) \in \left[ -1+\varepsilon
,1-\varepsilon \right] ^{2}.$

On the other hand,%
\begin{eqnarray}
&&\frac{1}{n^{3/4}}\sum\limits_{n/2\leq j<n-1}\frac{\left\vert g\left(
s+s_{j},t\right) -g\left( s_{j},t\right) \right\vert }{\left( n-j\right)
^{1/4}}  \label{2} \\
&\leq &\sum\limits_{n/2\leq j<n-1}\frac{\left\vert g\left( s+s_{j},t\right)
-g\left( s_{j},t\right) \right\vert }{n-j}  \notag \\
&\leq &c\min\limits_{1\leq k<n}\left\{ \omega _{1}\left( f;\frac{1}{n}%
\right) \log \left( k+1\right) +\left\{ j+k\right\} V_{1}\left(
f;I^{2}\right) \right\}  \notag \\
&=&o\left( 1\right) \text{ \ \ as \ \ }n\rightarrow \infty  \notag
\end{eqnarray}%
uniformly with respect to $\left( x,y\right) \in \left[ -1+\varepsilon
,1-\varepsilon \right] ^{2}.$

Combining (\ref{II41a})-(\ref{2}) we obtain that%
\begin{equation}
II_{41}=o\left( 1\right) \text{ \ \ as \ \ }n\rightarrow \infty  \label{II41}
\end{equation}%
uniformly with respect to $\left( x,y\right) \in \left[ -1+\varepsilon
,1-\varepsilon \right] ^{2}.$

Applying the Abel's transformation we obtain%
\begin{eqnarray}
&&II_{43} \label{II431+II432} \\
&=&\int\limits_{t_{1}}^{y}\left( g\left( s_{1},t\right)
\sum\limits_{k=1}^{n-1}\int\limits_{s_{k}}^{s_{k+1}}K_{n}\left( x,s\right)
ds\right) K_{m}\left( y,t\right) dt  \notag   \\
&&+\int\limits_{t_{1}}^{y}\left( \sum\limits_{j=1}^{n-2}\left( g\left(
s_{j+1},t\right) -g\left( s_{j},t\right) \right)
\sum\limits_{k=j+1}^{n-1}\int\limits_{s_{k}}^{s_{k+1}}K_{n}\left( x,s\right)
ds\right) K_{m}\left( y,t\right) dt  \notag \\
&=&\int\limits_{t_{1}}^{y}\left( g\left( s_{1},t\right)
\int\limits_{s_{1}}^{1}K_{n}\left( x,s\right) ds\right) K_{m}\left(
y,t\right) dt  \notag \\
&&+\int\limits_{t_{1}}^{y}\left( \sum\limits_{j=1}^{n-2}\left( g\left(
s_{j+1},t\right) -g\left( s_{j},t\right) \right)
\int\limits_{s_{j+1}}^{1}K_{n}\left( x,s\right) ds\right) K_{m}\left(
y,t\right) dt  \notag \\
&=&II_{431}+II_{432}.  \notag
\end{eqnarray}

It is easy to show that%
\begin{eqnarray}
\left\vert II_{431}\right\vert &\leq &\frac{c\left( \varepsilon \right) m}{%
n\left( s_{1}-x\right) }\int\limits_{t_{1}}^{y}\left\vert f\left(
s_{1},t\right) -f\left( x,y\right) \right\vert dt  \label{II431} \\
&\leq &c\left( \varepsilon \right) m\int\limits_{t_{1}}^{y}\left\vert
f\left( s_{1},t\right) -f\left( s_{1},y\right) \right\vert dt  \notag \\
&&+c\left( \varepsilon \right) m\int\limits_{t_{1}}^{y}\left\vert f\left(
s_{1},y\right) -f\left( x,y\right) \right\vert dt  \notag \\
&\leq &c\left( \varepsilon \right) \left\{ \omega _{1}\left( f,\frac{1}{n}%
\right) +\omega _{2}\left( f,\frac{1}{m}\right) \right\} =o\left( 1\right)
\text{ \ as \ }n,m\rightarrow \infty  \notag
\end{eqnarray}%
uniformly with respect to $\left( x,y\right) \in \left[ -1+\varepsilon
,1-\varepsilon \right] ^{2}.$

From (\ref{Kn^}), (\ref{1}) and (\ref{2}) we obtain%
\begin{eqnarray}
\left\vert II_{432}\right\vert &\leq &c\left( \varepsilon \right)
\int\limits_{t_{1}}^{y}\sum\limits_{j=1}^{n-2}\frac{\left\vert g\left(
s_{j+1},t\right) -g\left( s_{j},t\right) \right\vert }{\left(
s_{j+1}-x\right) n}\left\vert K_{m}\left( y,t\right) \right\vert dt
\label{II432} \\
&\leq &c\left( \varepsilon \right) \sup\limits_{t\in \left[ t_{1},y\right]
}\sum\limits_{j=1}^{n-2}\frac{\left\vert g\left( s_{j+1},t\right) -g\left(
s_{j},t\right) \right\vert }{j}  \notag \\
&=&o\left( 1\right) \text{ \ as \ }n,m\rightarrow \infty  \notag
\end{eqnarray}%
uniformly with respect to $\left( x,y\right) \in \left[ -1+\varepsilon
,1-\varepsilon \right] ^{2}.$

From (\ref{II431+II432})-(\ref{II432}) we have%
\begin{equation}
II_{43}=o\left( 1\right) \text{ \ as \ }n,m\rightarrow \infty  \label{II43}
\end{equation}%
uniformly with respect to $\left( x,y\right) \in \left[ -1+\varepsilon
,1-\varepsilon \right] ^{2}.$

Combining (\ref{II41-II43}), (\ref{II41}), (\ref{II42}) and (\ref{II43}) we
conclude that%
\begin{equation}
II_{4}=o\left( 1\right) \text{ \ as \ }n,m\rightarrow \infty  \label{II4}
\end{equation}%
uniformly with respect to $\left( x,y\right) \in \left[ -1+\varepsilon
,1-\varepsilon \right] ^{2}.$

Analogously, we can prove that%
\begin{equation}
II_{1}=o\left( 1\right) \text{ \ as \ }n,m\rightarrow \infty  \label{II1}
\end{equation}%
uniformly with respect to $\left( x,y\right) \in \left[ -1+\varepsilon
,1-\varepsilon \right] ^{2}.$

For $II_{2}$ we can write%
\begin{eqnarray}
\left\vert II_{2}\right\vert &\leq
&\int\limits_{x}^{s_{1}}\int\limits_{t_{1}}^{y}\left\vert f\left( s,t\right)
-f\left( x,y\right) \right\vert \left\vert K_{n}\left( x,s\right)
\right\vert \left\vert K_{m}\left( y,t\right) \right\vert dsdt  \label{II2}
\\
&\leq &c\left( \varepsilon \right) \left\{ \omega _{1}\left( f,\frac{1}{n}%
\right) +\omega _{2}\left( f,\frac{1}{m}\right) \right\} =o\left( 1\right)
\text{ \ as \ }n,m\rightarrow \infty  \notag
\end{eqnarray}%
uniformly with respect to $\left( x,y\right) \in \left[ -1+\varepsilon
,1-\varepsilon \right] ^{2}.$

We can write%
\begin{eqnarray}
&&II_{3} \label{III1-III4} \\
&=&\sum\limits_{j=1}^{n-1}\sum\limits_{i=1}^{m-1}\int%
\limits_{s_{j}}^{s_{j+1}}\int\limits_{t_{i+1}}^{t_{i}}g\left( s,t\right)
K_{n}\left( x,s\right) K_{m}\left( y,t\right) dsdt  \notag  
\\
&=&\sum\limits_{j=1}^{n-1}\sum\limits_{i=1}^{m-1}\int%
\limits_{s_{j}}^{s_{j+1}}\int\limits_{t_{i+1}}^{t_{i}}\left( g\left(
s,t\right) -g\left( s_{j},t\right) -g\left( s,t_{i}\right) +g\left(
s_{j},t_{i}\right) \right)  \notag \\
&&\times K_{n}\left( x,s\right) K_{m}\left( y,t\right) dsdt  \notag \\
&&+\sum\limits_{j=1}^{n-1}\sum\limits_{i=1}^{m-1}\int%
\limits_{s_{j}}^{s_{j+1}}\int\limits_{t_{i+1}}^{t_{i}}\left( g\left(
s_{j},t\right) -g\left( s_{j},t_{i}\right) \right) K_{n}\left( x,s\right)
K_{m}\left( y,t\right) dsdt  \notag \\
&&+\sum\limits_{j=1}^{n-1}\sum\limits_{i=1}^{m-1}\int%
\limits_{s_{j}}^{s_{j+1}}\int\limits_{t_{i+1}}^{t_{i}}\left( g\left(
s,t_{i}\right) -g\left( s_{j},t_{i}\right) \right) K_{n}\left( x,s\right)
K_{m}\left( y,t\right) dsdt  \notag \\
&&+\sum\limits_{j=1}^{n-1}\sum\limits_{i=1}^{m-1}g\left( s_{j},t_{i}\right)
\int\limits_{s_{j}}^{s_{j+1}}\int\limits_{t_{i+1}}^{t_{i}}K_{n}\left(
x,s\right) K_{m}\left( y,t\right) dsdt  \notag \\
&=&III_{1}+III_{2}+III_{3}+III_{4}.  \notag
\end{eqnarray}

For $III_{3}$ we have%
\begin{eqnarray}
&&III_{3} \label{III31+III32} \\
&=&\sum\limits_{j=1}^{n-2}\sum\limits_{i=1}^{m-1}\int%
\limits_{s_{j}}^{s_{j+1}}\int\limits_{t_{i+1}}^{t_{i}}\left( g\left(
s,t_{i}\right) -g\left( s_{j},t_{i}\right) \right) K_{n}\left( x,s\right)
K_{m}\left( y,t\right) dsdt  \notag   \\
&&+\sum\limits_{i=1}^{m-1}\int\limits_{s_{n-1}}^{1}\int%
\limits_{t_{i+1}}^{t_{i}}\left( g\left( s,t_{i}\right) -g\left(
s_{n-1},t_{i}\right) \right) K_{n}\left( x,s\right) K_{m}\left( y,t\right)
dsdt  \notag \\
&=&III_{31}+III_{32}.  \notag
\end{eqnarray}

Applying the Abel's transformation we get%
\begin{eqnarray}
&&III_{31} \label{III311+III312} \\
&=&\sum\limits_{j=1}^{n-2}\sum\limits_{i=1}^{m-1}\int%
\limits_{s_{j}}^{s_{j+1}}\int\limits_{t_{i+1}}^{t_{i}}\left( f\left(
s,t_{1}\right) -f\left( s_{j},t_{1}\right) \right) K_{n}\left( x,s\right)
K_{m}\left( y,t\right) dsdt  \notag   \\
&&+\sum\limits_{j=1}^{n-2}\sum\limits_{i=1}^{m-2}\sum\limits_{k=i+1}^{m-1}%
\int\limits_{s_{j}}^{s_{j+1}}\int\limits_{t_{k+1}}^{t_{k}}\left( f\left(
s,t_{i+1}\right) -f\left( s_{j},t_{i+1}\right) \right.   \notag \\
&&\left. -f\left( s,t_{i}\right) +f\left( s_{j},t_{i}\right) \right)
K_{n}\left( x,s\right) K_{m}\left( y,t\right) dsdt  \notag \\
&=&\sum\limits_{j=1}^{n-2}\int\limits_{s_{j}}^{s_{j+1}}\int%
\limits_{-1}^{t_{1}}\left( f\left( s,t_{1}\right) -f\left(
s_{j},t_{1}\right) \right) K_{n}\left( x,s\right) K_{m}\left( y,t\right) dsdt
\notag \\
&&+\sum\limits_{j=1}^{n-2}\sum\limits_{i=1}^{m-2}\int%
\limits_{s_{j}}^{s_{j+1}}\int\limits_{-1}^{t_{i+1}}\left( f\left(
s,t_{i+1}\right) -f\left( s_{j},t_{i+1}\right) \right.   \notag \\
&&\left. -f\left( s,t_{i}\right) +f\left( s_{j},t_{i}\right) K_{n}\left(
x,s\right) \right) K_{m}\left( y,t\right) dsdt  \notag \\
&=&III_{311}+III_{312}.  \notag
\end{eqnarray}

From (\ref{Kn}), (\ref{Kn_}), (\ref{1}) and (\ref{2}) we obtain%
\begin{eqnarray}
&&\left\vert III_{311}\right\vert  \label{III311} \\
&\leq &c\left( \varepsilon \right) \left\vert
\int\limits_{-1}^{t_{1}}K_{m}\left( y,t\right) dt\right\vert
\sum\limits_{j=1}^{n-2}\int\limits_{s_{j}}^{s_{j+1}}\frac{\left\vert f\left(
s,t_{1}\right) -f\left( s_{j},t_{1}\right) \right\vert }{\left(
s_{j}-x\right) \left( 1-s_{j+1}\right) ^{1/4}\left( 1+s_{j}\right) ^{1/4}}ds
\notag \\
&\leq &\frac{c\left( \varepsilon \right) n^{5/4}}{m\left( y-t_{1}\right) }%
\sum\limits_{j=1}^{n-2}\int\limits_{s_{j}}^{s_{j+1}}\frac{\left\vert f\left(
s,t_{1}\right) -f\left( s_{j},t_{1}\right) \right\vert }{j\left( n-j\right)
^{1/4}}ds  \notag \\
&=&c\left( \varepsilon \right) n^{5/4}\int\limits_{0}^{\frac{1-x}{n}%
}\sum\limits_{j=1}^{n-2}\frac{\left\vert f\left( s+s_{j},t_{1}\right)
-f\left( s_{j},t_{1}\right) \right\vert }{j\left( n-j\right) ^{1/4}}ds
\notag \\
&=&o\left( 1\right) \text{ \ \ as \ }n,m\rightarrow \infty  \notag
\end{eqnarray}%
uniformly with respect to $\left( x,y\right) \in \left[ -1+\varepsilon
,1-\varepsilon \right] ^{2},$%
\begin{eqnarray}
&&\left\vert III_{312}\right\vert  \label{III312+} \\
&\leq &\sum\limits_{j=1}^{n-2}\sum\limits_{i=1}^{m-2}\left\vert
\int\limits_{-1}^{t_{i+1}}K_{m}\left( y,t\right) dt\right\vert
\int\limits_{s_{j}}^{s_{j+1}}\left\vert f\left( s,t_{i+1}\right) -f\left(
s_{j},t_{i+1}\right) \right.  \notag \\
&&\left. -f\left( s,t_{i}\right) +f\left( s_{j},t_{i}\right) K_{n}\left(
x,s\right) \right\vert ds  \notag \\
&\leq &c\left( \varepsilon \right)
\sum\limits_{j=1}^{n-2}\sum\limits_{i=1}^{m-2}\frac{1}{m\left(
y-t_{i+1}\right) }\int\limits_{s_{j}}^{s_{j+1}}\left\vert f\left(
s,t_{i+1}\right) -f\left( s_{j},t_{i+1}\right) \right.  \notag \\
&&\left. -f\left( s,t_{i}\right) +f\left( s_{j},t_{i}\right) \right\vert
\frac{ds}{\left( s_{j}-x\right) \left( 1-s_{j+1}\right) ^{1/4}}  \notag \\
&\leq &c\left( \varepsilon \right)
\sum\limits_{j=1}^{n-2}\sum\limits_{i=1}^{m-2}\frac{1}{i}\frac{n^{5/4}}{%
j\left( n-j\right) ^{1/4}}\int\limits_{s_{j}}^{s_{j+1}}\left\vert f\left(
s,t_{i+1}\right) -f\left( s_{j},t_{i+1}\right) \right.  \notag \\
&&\left. -f\left( s,t_{i}\right) +f\left( s_{j},t_{i}\right) \right\vert ds
\notag \\
&=&c\left( \varepsilon \right) \int\limits_{0}^{\frac{1-x}{n}%
}\sum\limits_{j=1}^{n-2}\sum\limits_{i=1}^{m-2}\frac{1}{i}\frac{n^{5/4}}{%
j\left( n-j\right) ^{1/4}}\left\vert f\left( s+s_{j},t_{i+1}\right) -f\left(
s_{j},t_{i+1}\right) \right.  \notag \\
&&\left. -f\left( s+s_{j},t_{i}\right) +f\left( s_{j},t_{i}\right)
\right\vert ds  \notag \\
&=&c\left( \varepsilon \right) \int\limits_{0}^{\frac{1-x}{n}%
}\sum\limits_{1\leq j<n/2}\sum\limits_{i=1}^{m-2}\frac{n^{5/4}}{ji\left(
n-j\right) ^{1/4}}\left\vert f\left( s+s_{j},t_{i+1}\right) -f\left(
s_{j},t_{i+1}\right) \right.  \notag \\
&&\left. -f\left( s+s_{j},t_{i}\right) +f\left( s_{j},t_{i}\right)
\right\vert ds  \notag \\
&&+c\left( \varepsilon \right) \int\limits_{0}^{\frac{1-x}{n}%
}\sum\limits_{n/2\leq j<n-1}\sum\limits_{i=1}^{m-2}\frac{n^{5/4}}{ji\left(
n-j\right) ^{1/4}}\left\vert f\left( s+s_{j},t_{i+1}\right) -f\left(
s_{j},t_{i+1}\right) \right.  \notag \\
&&\left. -f\left( s+s_{j},t_{i}\right) +f\left( s_{j},t_{i}\right)
\right\vert ds  \notag \\
&\leq &c\left( \varepsilon \right) n\int\limits_{0}^{\frac{1-x}{n}%
}\sum\limits_{1\leq j<n/2}\sum\limits_{i=1}^{m-2}\frac{1}{ji}\left\vert
f\left( s+s_{j},t_{i+1}\right) -f\left( s_{j},t_{i+1}\right) \right.  \notag
\\
&&\left. -f\left( s+s_{j},t_{i}\right) +f\left( s_{j},t_{i}\right)
\right\vert ds  \notag \\
&&+c\left( \varepsilon \right) n\int\limits_{0}^{\frac{1-x}{n}%
}\sum\limits_{n/2\leq j<n-1}\sum\limits_{i=1}^{m-2}\frac{1}{\left(
n-j\right) i}\left\vert f\left( s+s_{j},t_{i+1}\right) -f\left(
s_{j},t_{i+1}\right) \right.  \notag \\
&&\left. -f\left( s+s_{j},t_{i}\right) +f\left( s_{j},t_{i}\right)
\right\vert ds  \notag \\
&\leq &\min\limits_{1\leq k<n}\min\limits_{1\leq l<m}\left\{ \omega
_{1,2}\left( f;\frac{1}{n},\frac{1}{m}\right) \log \left( k+1\right) \log
\left( l+1\right) \right.  \notag \\
&&+\left\{ i+k\right\} \left\{ j\right\} V_{1,2}\left( f;I^{2}\right)  \notag
\\
&&\left. +\left\{ i\right\} \left\{ j+l\right\} V_{1,2}\left( f;I^{2}\right)
\right\}  \notag \\
&=&o\left( 1\right) \text{ \ \ as \ }n,m\rightarrow \infty  \notag
\end{eqnarray}%
uniformly with respect to $\left( x,y\right) \in \left[ -1+\varepsilon
,1-\varepsilon \right] ^{2}.$

Combining (\ref{III311+III312}), (\ref{III311}) and (\ref{III312+}) we have%
\begin{equation}
III_{31}=o\left( 1\right) \text{ \ \ as \ }n,m\rightarrow \infty
\label{III31}
\end{equation}%
uniformly with respect to $\left( x,y\right) \in \left[ -1+\varepsilon
,1-\varepsilon \right] ^{2}.$

Applying the Abel's transformation we obtain%
\begin{eqnarray}
&&III_{32}  \label{III321+III322} \\
&=&\sum\limits_{i=1}^{m-1}\int\limits_{s_{n-1}}^{1}\int%
\limits_{t_{i+1}}^{t_{i}}\left( f\left( s,t_{i}\right) -f\left(
s_{n-1},t_{i}\right) \right) K_{n}\left( x,s\right) K_{m}\left( y,t\right)
dsdt  \notag \\
&=&\sum\limits_{i=1}^{m-1}\int\limits_{s_{n-1}}^{1}\int%
\limits_{t_{i+1}}^{t_{i}}\left( f\left( s,t_{1}\right) -f\left(
s_{n-1},t_{1}\right) \right) K_{n}\left( x,s\right) K_{m}\left( y,t\right)
dsdt  \notag \\
&&+\sum\limits_{i=1}^{m-2}\sum\limits_{k=i+1}^{m-1}\int\limits_{s_{n-1}}^{1}%
\int\limits_{t_{k+1}}^{t_{k}}\left( f\left( s,t_{i+1}\right) -f\left(
s_{n-1},t_{i+1}\right) \right.  \notag \\
&&\left. -f\left( s,t_{i}\right) +f\left( s_{n-1},t_{i}\right) \right)
K_{n}\left( x,s\right) K_{m}\left( y,t\right) dsdt  \notag \\
&=&\int\limits_{s_{n-1}}^{1}\int\limits_{-1}^{t_{m-1}}\left( f\left(
s,t_{1}\right) -f\left( s_{n-1},t_{1}\right) \right) K_{n}\left( x,s\right)
K_{m}\left( y,t\right) dsdt  \notag \\
&&+\sum\limits_{i=1}^{m-2}\int\limits_{s_{n-1}}^{1}\int%
\limits_{-1}^{t_{i+1}}\left( f\left( s,t_{i+1}\right) -f\left(
s_{n-1},t_{i+1}\right) \right.  \notag \\
&&\left. -f\left( s,t_{i}\right) +f\left( s_{n-1},t_{i}\right) \right)
K_{n}\left( x,s\right) K_{m}\left( y,t\right) dsdt  \notag \\
&=&III_{321}+III_{322}.  \notag
\end{eqnarray}

Since%
\begin{eqnarray*}
\int\limits_{s_{n-1}}^{1}\left\vert K_{n}\left( x,s\right) \right\vert ds
&\leq &c\left( \varepsilon \right) \int\limits_{s_{n-1}}^{1}\frac{ds}{\left(
s-x\right) \left( 1-s\right) ^{1/4}} \\
&\leq &c\left( \varepsilon \right) \int\limits_{s_{n-1}}^{1}\frac{ds}{\left(
1-s\right) ^{1/4}} \\
&\leq &\frac{c\left( \varepsilon \right) }{n^{3/4}}
\end{eqnarray*}%
and%
\begin{equation*}
\int\limits_{-1}^{t_{m-1}}\left\vert K_{m}\left( y,t\right) \right\vert
dt\leq \frac{c\left( \varepsilon \right) }{m^{3/4}}
\end{equation*}%
for $III_{321}$ we can write%
\begin{equation}
\left\vert III_{321}\right\vert \leq \frac{c\left( \varepsilon \right)
\left\Vert f\right\Vert _{C}}{\left( nm\right) ^{3/4}}=o\left( 1\right)
\text{ \ \ as \ }n,m\rightarrow \infty   \label{III321}
\end{equation}%
uniformly with respect to $\left( x,y\right) \in \left[ -1+\varepsilon
,1-\varepsilon \right] ^{2}.$

On the other hand,%
\begin{eqnarray}
\left\vert III_{322}\right\vert &\leq &\frac{c\left( \varepsilon \right) }{%
n^{3/4}}\sup\limits_{s}\sum\limits_{i=1}^{m-2}\left\vert f\left(
s,t_{i+1}\right) -f\left( s,t_{i}\right) \right\vert \left\vert
\int\limits_{-1}^{t_{i+1}}K_{m}\left( y,t\right) dt\right\vert
\label{III322} \\
&\leq &\frac{c\left( \varepsilon \right) }{n^{3/4}}\sup\limits_{s}\sum%
\limits_{i=1}^{m-2}\frac{\left\vert f\left( s,t_{i+1}\right) -f\left(
s,t_{i}\right) \right\vert }{m\left( t_{i+1}-x\right) }  \notag \\
&=&\frac{c\left( \varepsilon \right) }{n^{3/4}}\sup\limits_{s}\sum%
\limits_{i=1}^{m-2}\frac{\left\vert f\left( s,t_{i+1}\right) -f\left(
s,t_{i}\right) \right\vert }{i}  \notag \\
&\leq &\frac{c\left( \varepsilon \right) }{n^{3/4}}HV_{2}\left(
f,I^{2}\right) =o(1)\text{ \ as \ }n,m\rightarrow \infty  \notag
\end{eqnarray}%
uniformly with respect to $\left( x,y\right) \in \left[ -1+\varepsilon
,1-\varepsilon \right] ^{2}.$

From (\ref{III321+III322}), (\ref{III321}) and (\ref{III322}) we have%
\begin{equation}
III_{32}=o\left( 1\right) \text{ \ \ as \ }n,m\rightarrow \infty
\label{III32}
\end{equation}%
uniformly with respect to $\left( x,y\right) \in \left[ -1+\varepsilon
,1-\varepsilon \right] ^{2}.$

Combining (\ref{III31+III32}), (\ref{III31}) and (\ref{III32}) we conclude
that%
\begin{equation}
III_{3}=o\left( 1\right) \text{ \ \ as \ }n,m\rightarrow \infty  \label{III3}
\end{equation}%
uniformly with respect to $\left( x,y\right) \in \left[ -1+\varepsilon
,1-\varepsilon \right] ^{2}.$

Analogously we can prove that%
\begin{equation}
III_{2}=o\left( 1\right) \text{ \ \ as \ }n,m\rightarrow \infty  \label{III2}
\end{equation}%
uniformly with respect to $\left( x,y\right) \in \left[ -1+\varepsilon
,1-\varepsilon \right] ^{2}.$

From (\ref{Kn}) we have%
\begin{eqnarray}
&&\left\vert III_{1}\right\vert \label{III_1}  \\
&\leq &c\left( \varepsilon \right)
\sum\limits_{j=1}^{n-1}\sum\limits_{i=1}^{m-1}\int\limits_{s_{j}}^{s_{j+1}}%
\int\limits_{t_{i+1}}^{t_{i}}\left\vert f\left( s,t\right) -f\left(
s_{j},t\right) -f\left( s,t_{i}\right) +f\left( s_{j},t_{i}\right)
\right\vert \notag \\
&&\times \frac{1}{s-x}\frac{1}{y-t}\frac{1}{\left( 1-s\right) ^{1/4}}\frac{1%
}{\left( 1+t\right) ^{1/4}}dsdt \notag \\
&\leq &\left( nm\right) ^{5/4}\int\limits_{0}^{\frac{1-x}{n}%
}\int\limits_{0}^{\frac{1-y}{m}}\sum\limits_{j=1}^{n-1}\sum%
\limits_{i=1}^{m-1}\frac{1}{j\left( n-j\right) ^{1/4}}\frac{1}{i\left(
m-i\right) ^{1/4}} \notag \\
&&\times \left\vert f\left( s+s_{j},t+t_{i}\right) -f\left(
s_{j},t+t_{i}\right) -f\left( s+s_{j},t_{i}\right) +f\left(
s_{j},t_{i}\right) \right\vert dsdt.\notag
\end{eqnarray}%
We can write%
\begin{eqnarray}
&&\sqrt[4]{nm}\sum\limits_{j=1}^{n-1}\sum\limits_{i=1}^{m-1}\frac{1}{j\left(
n-j\right) ^{1/4}}\frac{1}{i\left( m-i\right) ^{1/4}}  \label{IV1-IV4} \\
&&\times \left\vert f\left( s+s_{j},t+t_{i}\right) -f\left(
s_{j},t+t_{i}\right) -f\left( s+s_{j},t_{i}\right) +f\left(
s_{j},t_{i}\right) \right\vert   \notag \\
&\leq &\sum\limits_{1\leq j<n/2}\sum\limits_{1\leq i<m/2}\frac{1}{ji}  \notag
\\
&&\times \left\vert f\left( s+s_{j},t+t_{i}\right) -f\left(
s_{j},t+t_{i}\right) -f\left( s+s_{j},t_{i}\right) +f\left(
s_{j},t_{i}\right) \right\vert   \notag \\
&&+\frac{1}{m^{3/4}}\sum\limits_{1\leq j<n/2}\sum\limits_{m/2\leq i<m-1}%
\frac{1}{j\left( m-i\right) ^{1/4}}  \notag \\
&&\times \left\vert f\left( s+s_{j},t+t_{i}\right) -f\left(
s_{j},t+t_{i}\right) -f\left( s+s_{j},t_{i}\right) +f\left(
s_{j},t_{i}\right) \right\vert   \notag \\
&&+\frac{1}{n^{3/4}}\sum\limits_{n/2\leq j<n-1}\sum\limits_{1\leq i<m/2}%
\frac{1}{\left( n-j\right) ^{1/4}i}  \notag \\
&&\times \left\vert f\left( s+s_{j},t+t_{i}\right) -f\left(
s_{j},t+t_{i}\right) -f\left( s+s_{j},t_{i}\right) +f\left(
s_{j},t_{i}\right) \right\vert   \notag \\
&&+\frac{1}{\left( nm\right) ^{3/4}}\sum\limits_{n/2\leq
j<n-1}\sum\limits_{m/2\leq i<m-1}\frac{1}{\left( n-j\right) ^{1/4}\left(
m-j\right) ^{1/4}} \notag \\
&&\times \left\vert f\left( s+s_{j},t+t_{i}\right) -f\left(
s_{j},t+t_{i}\right) -f\left( s+s_{j},t_{i}\right) +f\left(
s_{j},t_{i}\right) \right\vert \notag \\
&\leq &\sum\limits_{1\leq j<n/2}\sum\limits_{1\leq i<m/2}\frac{1}{ji}  \notag
\\
&&\times \left\vert f\left( s+s_{j},t+t_{i}\right) -f\left(
s_{j},t+t_{i}\right) -f\left( s+s_{j},t_{i}\right) +f\left(
s_{j},t_{i}\right) \right\vert   \notag \\
&&+\sum\limits_{1\leq j<n/2}\sum\limits_{m/2\leq i<m-1}\frac{1}{j\left(
m-i\right) }  \notag \\
&&\times \left\vert f\left( s+s_{j},t+t_{i}\right) -f\left(
s_{j},t+t_{i}\right) -f\left( s+s_{j},t_{i}\right) +f\left(
s_{j},t_{i}\right) \right\vert   \notag \\
&&+\sum\limits_{n/2\leq j<n-1}\sum\limits_{1\leq i<m/2}\frac{1}{\left(
n-j\right) i}  \notag \\
&&\times \left\vert f\left( s+s_{j},t+t_{i}\right) -f\left(
s_{j},t+t_{i}\right) -f\left( s+s_{j},t_{i}\right) +f\left(
s_{j},t_{i}\right) \right\vert   \notag \\
&&+\sum\limits_{n/2\leq j<n-1}\sum\limits_{m/2\leq i<m-1}\frac{1}{\left(
n-j\right) \left( m-j\right) }\notag \\
&&\times \left\vert f\left( s+s_{j},t+t_{i}\right) -f\left(
s_{j},t+t_{i}\right) -f\left( s+s_{j},t_{i}\right) +f\left(
s_{j},t_{i}\right) \right\vert  \notag\\
&=&IV_{1}+IV_{2}+IV_{3}+IV_{4}.  \notag
\end{eqnarray}

It is easy to show that%
\begin{eqnarray}
&&IV_{1}  \label{IV1} 
\leq \min\limits_{1\leq l<n}\min\limits_{1\leq r<m}\left\{ \omega
_{12}\left( f;\frac{1}{n},\frac{1}{m}\right) \log \left( l+1\right) \log
\left( r+1\right) \right.  \\
&&\left. +\left\{ i+l\right\} \left\{ j\right\} V_{1,2}\left( f;I^{2}\right)
+\left\{ i\right\} \left\{ j+r\right\} V_{1,2}\left( f;I^{2}\right) \right\}
\notag \\
&=&o\left( 1\right) \text{ \ \ as \ }n,m\rightarrow \infty  \notag
\end{eqnarray}%
uniformly with respect to $\left( x,y\right) \in \left[ -1+\varepsilon
,1-\varepsilon \right] ^{2}.$

Analogously, we can prove that%
\begin{equation}
IV_{i}=o\left( 1\right) \text{ \ \ as \ }n,m\rightarrow \infty ,i=2,3,4
\label{IV2-IV4}
\end{equation}%
uniformly with respect to $\left( x,y\right) \in \left[ -1+\varepsilon
,1-\varepsilon \right] ^{2}.$

Combining (\ref{III_1}), (\ref{IV1-IV4}), (\ref{IV1}) and (\ref{IV2-IV4}) we get%
\begin{equation}
III_{1}=o\left( 1\right) \text{ \ \ as \ }n,m\rightarrow \infty  \label{III1}
\end{equation}%
uniformly with respect to $\left( x,y\right) \in \left[ -1+\varepsilon
,1-\varepsilon \right] ^{2}.$

Finally, we estimate $III_{4}$. By Abel's transformation we have%
\begin{eqnarray}
&&III_{4}  \label{III41-III44} \\
&=&\sum\limits_{j=1}^{n-2}\sum\limits_{i=1}^{m-2}\left( f\left(
s_{j},t_{i}\right) -f\left( s_{j+1},t_{i}\right) -f\left(
s_{j},t_{i+1}\right) +f\left( s_{j+1},t_{i+1}\right) \right)  \notag \\
&&\times
\sum\limits_{k=j+1}^{n-1}\sum\limits_{l=i+1}^{m-1}\int%
\limits_{s_{k}}^{s_{k+1}}\int\limits_{t_{l+1}}^{t_{l}}K_{n}\left( x,s\right)
K_{m}\left( y,t\right) dsdt  \notag \\
&&+\sum\limits_{i=1}^{m-2}\left( f\left( s_{1},t_{i+1}\right) -f\left(
s_{1},t_{i}\right) \right)
\sum\limits_{l=i+1}^{m-1}\int\limits_{s_{0}}^{s_{1}}\int%
\limits_{t_{l+1}}^{t_{l}}K_{n}\left( x,s\right) K_{m}\left( y,t\right) dsdt
\notag \\
&&+\sum\limits_{j=1}^{n-2}\left( f\left( s_{j+1},t_{1}\right) -f\left(
s_{j},t_{1}\right) \right)
\sum\limits_{k=j+1}^{n-1}\int\limits_{s_{k}}^{s_{k+1}}\int%
\limits_{t_{1}}^{t_{0}}K_{n}\left( x,s\right) K_{m}\left( y,t\right) dsdt
\notag \\
&&+g\left( s_{1},t_{1}\right)
\sum\limits_{j=1}^{n-1}\sum\limits_{i=1}^{m-1}\int\limits_{s_{j}}^{s_{j+1}}%
\int\limits_{t_{i+1}}^{t_{i}}K_{n}\left( x,s\right) K_{m}\left( y,t\right)
dsdt  \notag \\
&=&III_{41}+III_{42}+III_{43}+III_{44}.  \notag
\end{eqnarray}

From (\ref{Kn_}), (\ref{Kn^}), (\ref{1}), (\ref{2}) and (\ref{IV1-IV4}) we
have%
\begin{eqnarray}
&&\left\vert III_{41}\right\vert  \label{III41} \\
&\leq &\sum\limits_{j=1}^{n-2}\sum\limits_{i=1}^{m-2}\left\vert f\left(
s_{j},t_{i}\right) -f\left( s_{j+1},t_{i}\right) -f\left(
s_{j},t_{i+1}\right) +f\left( s_{j+1},t_{i+1}\right) \right\vert  \notag \\
&&\times \left\vert
\int\limits_{s_{j+1}}^{1}\int\limits_{-1}^{t_{i}}K_{n}\left( x,s\right)
K_{m}\left( y,t\right) dsdt\right\vert  \notag \\
&\leq &\sum\limits_{j=1}^{n-2}\sum\limits_{i=1}^{m-2}\left\vert f\left(
s_{j},t_{i}\right) -f\left( s_{j+1},t_{i}\right) -f\left(
s_{j},t_{i+1}\right) +f\left( s_{j+1},t_{i+1}\right) \right\vert  \notag \\
&&\times \frac{1}{n\left( s_{j+1}-x\right) }\frac{1}{m\left( y-t_{i}\right) }
\notag \\
&\leq &\sum\limits_{j=1}^{n-2}\sum\limits_{i=1}^{m-2}\frac{\left\vert
f\left( s_{j},t_{i}\right) -f\left( s_{j+1},t_{i}\right) -f\left(
s_{j},t_{i+1}\right) +f\left( s_{j+1},t_{i+1}\right) \right\vert }{ij}
\notag \\
&=&o\left( 1\right) \text{ \ \ as \ }n,m\rightarrow \infty  \notag
\end{eqnarray}%
uniformly with respect to $\left( x,y\right) \in \left[ -1+\varepsilon
,1-\varepsilon \right] ^{2}.$

\begin{eqnarray}
\left\vert III_{42}\right\vert  &\leq &\sum\limits_{i=1}^{m-2}\left\vert
f\left( s_{1},t_{i+1}\right) -f\left( s_{1},t_{i}\right) \right\vert
\label{III42} \\
&&\times \int\limits_{x}^{x+\frac{1-x}{n}}\left\vert K_{n}\left( x,s\right)
\right\vert ds\left\vert \int\limits_{-1}^{t_{i+1}}K_{m}\left( y,t\right)
dt\right\vert   \notag \\
&\leq &c\left( \varepsilon \right) \sum\limits_{i=1}^{m-2}\frac{\left\vert
f\left( s_{1},t_{i+1}\right) -f\left( s_{1},t_{i}\right) \right\vert }{i}
\notag \\
&=&o\left( 1\right) \text{ \ \ as \ }n,m\rightarrow \infty   \notag
\end{eqnarray}%
uniformly with respect to $\left( x,y\right) \in \left[ -1+\varepsilon
,1-\varepsilon \right] ^{2}.$

Analogously, we can prove that%
\begin{equation}
III_{43}=o\left( 1\right) \text{ \ \ as \ }n,m\rightarrow \infty
\label{III43}
\end{equation}%
uniformly with respect to $\left( x,y\right) \in \left[ -1+\varepsilon
,1-\varepsilon \right] ^{2},$%
\begin{eqnarray}
\left\vert III_{44}\right\vert &\leq &\left\vert f\left( s_{1},t_{1}\right)
-f\left( x,y\right) \right\vert \left\vert
\int\limits_{s_{1}}^{1}K_{m}\left( y,t\right) dt\right\vert \left\vert
\int\limits_{-1}^{t_{1}}K_{n}\left( x,s\right) ds\right\vert  \label{III44}
\\
&\leq &c\left( \varepsilon \right) \left( \omega _{1}\left( f;\frac{1}{n}%
\right) +\omega _{2}\left( f;\frac{1}{m}\right) \right)  \notag \\
&=&o\left( 1\right) \text{ \ \ as \ }n,m\rightarrow \infty  \notag
\end{eqnarray}%
uniformly with respect to $\left( x,y\right) \in \left[ -1+\varepsilon
,1-\varepsilon \right] ^{2}.$

From (\ref{III41-III44})-(\ref{III43}) we have%
\begin{equation}
III_{4}=o\left( 1\right) \text{ \ \ as \ }n,m\rightarrow \infty  \label{III4}
\end{equation}%
uniformly with respect to $\left( x,y\right) \in \left[ -1+\varepsilon
,1-\varepsilon \right] ^{2}.$

By (\ref{III1-III4}), (\ref{III3}), (\ref{III2}), (\ref{III1}) and (\ref%
{III4}) we obtain%
\begin{equation}
II_{3}=o\left( 1\right) \text{ \ \ as \ }n,m\rightarrow \infty  \label{II3}
\end{equation}%
uniformly with respect to $\left( x,y\right) \in \left[ -1+\varepsilon
,1-\varepsilon \right] ^{2}.$

From (\ref{II1-II4}), (\ref{II4}), (\ref{II1}), (\ref{II2}) and (\ref{II3})
we conclude that%
\begin{equation}
I_{1}=o\left( 1\right) \text{ \ \ as \ }n,m\rightarrow \infty  \label{I1}
\end{equation}%
uniformly with respect to $\left( x,y\right) \in \left[ -1+\varepsilon
,1-\varepsilon \right] ^{2}.$

Analogously we can prove that%
\begin{equation}
I_{i}=o\left( 1\right) \text{ \ \ as \ }n,m\rightarrow \infty ,i=2,3,4
\label{I2-I4}
\end{equation}%
uniformly with respect to $\left( x,y\right) \in \left[ -1+\varepsilon
,1-\varepsilon \right] ^{2}.$

Combining (\ref{I1-I4}), (\ref{I1}) and (\ref{I2-I4}) we complete the proof
of Theorem \ref{CBV}.
\end{proof}

\end{document}